\documentclass[12pt,a4paper,english]{article}
\usepackage[T1]{fontenc}
\usepackage[latin9]{inputenc}
\usepackage{float}
\usepackage{amsmath}
\usepackage{amsthm}
\usepackage{amssymb}

\makeatletter

\theoremstyle{plain}
\newtheorem{thm}{\protect\theoremname}[section]
\theoremstyle{definition}
\newtheorem{defn}[thm]{\protect\definitionname}
\theoremstyle{plain}
\newtheorem{lem}[thm]{\protect\lemmaname}
\theoremstyle{plain}
\newtheorem{prop}[thm]{\protect\propositionname}
\theoremstyle{remark}
\newtheorem{rem}[thm]{\protect\remarkname}

\RequirePackage{fix-cm}

\makeatother

\usepackage{babel}
\providecommand{\definitionname}{Definition}
\providecommand{\lemmaname}{Lemma}
\providecommand{\propositionname}{Proposition}
\providecommand{\remarkname}{Remark}
\providecommand{\theoremname}{Theorem}

\begin{document}
\title{\textbf{String-Averaging Methods for Best Approximation to Common
Fixed Point Sets of Operators: The Finite and Infinite Cases}}
\author{Yair Censor\thanks{Corresponding author: Yair Censor, Email: yair@math.haifa.ac.il.}
and Ariel Nisenbaum\\
Department of Mathematics\\
The University of Haifa\\
Haifa 3498838, Israel}
\date{December 22, 2020. Revised: March 5, 2021.}
\maketitle
\begin{abstract}
{\normalsize{}String-averaging is an algorithmic structure used when
handling a family of operators in situations where the algorithm at
hand requires to employ the operators in a specific order. Sequential
orderings are well-known and a simultaneous order means that all operators
are used simultaneously (in parallel). String-averaging allows to
use strings of indices, constructed by subsets of the index set of
all operators, to apply the operators along these strings and then
to combine their end-points in some agreed manner to yield the next
iterate of the algorithm. String-averaging methods were discussed
and used for solving the common fixed point problem or its important
special case of the convex feasibility problem. In this paper we propose
and investigate string-averaging methods for the problem of best approximation
to the common fixed point set of a family of operators. This problem
involves finding a point in the common fixed point set of a family
of operators that is closest to a given point, called an anchor point,
in contrast with the common fixed point problem that seeks any point
in the common fixed point set.}{\normalsize\par}

{\normalsize{}We construct string-averaging methods for solving the
best approximation problem to the common fixed points set of either
finite or infinite families of firmly nonexpansive operators in a
real Hilbert space. We show that the simultaneous Halpern-Lions-Wittman-Bauschke
algorithm, the Halpern-Wittman algorithm and the Combettes algorithm,
which were not labeled as string-averaging methods, are actually special
cases of these methods. Some of our string-averaging methods are labeled
as ``static'' because they use a fixed pre-determined set of strings.
Others are labeled as ``quasi-dynamic'' because they allow the choices
of strings to vary, between iterations, in a specific manner and belong
to a finite fixed pre-determined set of applicable strings. For the
problem of best approximation to the common fixed point set of a family
of operators, the full dynamic case that would allow strings to unconditionally
vary, between iterations, remains unsolved although it exists and
is validated in the literature for the convex feasibility problem
where it is called ``dynamic string-averaging''.\medskip{}
}{\normalsize\par}

\textbf{Keywords}: String-averaging; {\normalsize{}common fixed points;
best approximation problem; firmly nonexpansive operators; HLWB algorithm;
Halpern theorem; }projection methods.
\end{abstract}

\section{Introduction \label{sec:Introduction}}

String-averaging algorithmic structures are used for handling a family
of operators in situations where the algorithm needs to employ the
operators in a specific order. String-averaging allows to use strings
of indices, taken from the index set of all operators, to apply the
operators along these strings and to combine their end-points in some
agreed manner to yield the next iterate of the algorithm.

If a point $p$ of a set $C$ is nearest to a given point $x$ in
space, known as the anchor, then $p$ is a \textit{best approximation
to $x$ from $C$}. In the case when $C$ is the common fixed points
set of a family of self-mapping operators, the problem of finding
such a $p$ is known as the \textit{best approximation problem} (BAP).
In case the fixed point sets are all convex, this problem\textbf{
}is a special case of two well-known problems:\textbf{ }\textit{The
convex feasibility problem} (CFP), which is to find a (any) point
in the intersection of closed convex sets, and the \textit{common
fixed point problem} (CFPP), where the closed convex sets in CFP are
the fixed point sets of operators of a given family. The CFP, the
CFPP and BAP are widely studied, and are useful in mathematics and
various physical science (see, e.g., Bauschke and Borwein \cite{BauschkeBorwein},
Reich and Zalas \cite{ReichZalas} and Cegielski \cite{Cegielski},
to name but a few), including in the context of string-averaging for
the CFP and the CFPP (see, e.g., Censor and Zaslavski \cite{CensorZaslavski1,CensorZaslavski2},
Bargetz, Reich and Zalas \cite{Bargetz} and \cite{ReichZalas}).

Nevertheless, string-averaging algorithmic approaches for solving
the BAP were, to the best of our knowledge, not proposed neither investigated.
Motivated by this, we devote our research presented here, to developing
string-averaging algorithmic schemes for finding the best approximation
to the common fixed points set of a, either finite or infinite, family
of firmly nonexpansive operators. Besides pure mathematical interest,
see, e.g., Dye, Khamsi and Reich \cite{reich1991} and references
therein, motivation for the infinite case can also come from practical
real-world situations. This point is succinctly made in the introduction
of the recent paper of Kong, Pajoohesh and Herman \cite{kong2019}.
Although they refer to the infinite CFP their arguments can serve
also to justify the infinite case for the BAP. The case of infinitely
many sets is useful in applications where there is a potential infinity
of samples or measurements, each one of which gives rise to a convex
set that contains the point we wish to recover, see, for example Blat
and Hero \cite{blat2006}.

\subsection{Contribution and structure of the paper}

Our string-averaging methods are applied to families of firmly nonexpansive
operators and they follow the principles of the original string-averaging
process suggested by Censor, Elfving and Herman in \cite{Pink-book}.
Consider a set $\mathcal{M}$ containing all pairs of the form $(\varOmega,w)$,
where $\varOmega$ is a set of finite strings of indices, by which
the string operators are formed, and $w$ is a function that attaches
to every string $t\in\varOmega$ a positive real weight $w(t)$ such
that the sum of weights equals $1$.

We construct what we call a ``static'' string-averaging methods
in which a fixed pre-determined single pair $(\varOmega,w)\in\mathcal{M}$
is used throughout the iterative process. This approach solves the
BAP for a common fixed points set of either a finite or an infinite
family of firmly nonexpansive operators and we show that the simultaneous
Halpern-Lions-Wittman-Bauschke algorithm (see, e.g., in Censor \cite[Algorithm 5]{Censor}),
the Halpern-Wittman algorithm (see, e.g., \cite[Algorithm 4.1]{BauschkeBorwein})
and the Combettes algorithm \cite{Combettes}, which were labeled
as sequential and simultaneous algorithms, respectively, are special
cases of our static string-averaging methods.

We extend the ``static'' notion in the finite case to the situation
where a finite number of pairs $(\varOmega,w)\in\mathcal{M}$ can
be used, and call it a ``quasi'' dynamic string-averaging method
due to its resemblance to the dynamic string-averaging scheme (see
\cite{CensorZaslavski1}). This is done by ordering a finite pre-defined
set of pairs $(\varOmega,w)\in\mathcal{M}$, and forming a finite
family of operators of the form $\sum_{t\in\Omega}w(t)T[t](x)$, for
every such $(\varOmega,w)$, where $T[t]$ are string operators. These
are, in turn, used in a cyclic manner.

With the aid of a construction similar to the quasi dynamic string-averaging
algorithm, we propose a simultaneous string-averaging method.

We focus here on Halpern's algorithm and the Halpern-Lions-Wittman-Bauschke
algorithm. There are, however, several other iterative processes for
solving the BAP, which are not treated here, but can possibly be also
extended via the string-averaging algorithmic concept, e.g., Dykstra's
algorithm, see \cite[Subsection 30.2]{BauschkeCombettes} and Haugazeau's
method, see \cite[Subsection 30.3]{BauschkeCombettes}, where further
references can be found.

The results presented here for the infinite case complement our earlier
work in Aleyner and Censor \cite{Aleyner} where a sequential algorithm
for solving the BAP to the common fixed points set of a semigroup
of nonexpansive operators in Hilbert space was studied. Only a sequential
iteartive process was investigated there but the framework was more
general due to the kind of operators (nonexpansive) and the size of
the pool of opeartors (a semigroup, not limited to the countable case).

This is a theoretical work in the spirit of the theoretical developments
in fixed point theory presented in many of the earlier referenced
papers. The string-averaging approach is actually not a single algorithm
but an ``algorithmic scheme'' so that every individual choice of
lengths and assignments of strings will give rise to a different algorithm.
In this way our string-averaging algorithmic scheme generalizes earlier
algorithms that become special cases of it. It is not practical to
conduct a numerical experiment without having a specific scientific
or real-world problem in hand but it is expected that researchers
who need to solve the best approximation problem to common fixed point
sets will find here valuable algorithmic information.

The paper is structured as follows. After preliminaries in Section
\ref{sec: Preliminaries}, the work is divided into two main parts:
A first part where the given family of firmly nonexpansive operators
is finite, in Sections \ref{sec:The-finite-case} and \ref{sec:Special case - orthogonal projctions},
and a second part where the given family of firmly nonexpansive operators
is a countable family, in Section \ref{sec:The-infinite-case}. The
static, the quasi-dynamic and the simultaneous string-averaging methods
can be found in Subsections \ref{subsec:static-SA-finite}, \ref{subsec:The-quasi-dynamic}
and \ref{subsec: The simultaneous QDSA}, respectively. In Subsection
\ref{sec:Special case - orthogonal projctions}, we show for which
choices of pairs ($\varOmega,w)$, the simultaneous version of the
Halpern-Lions-Wittman-Baucschke algorithm and the Halpern-Wittman
algorithm are special cases of our static string-averaging approach.
In Subsection \ref{subsec:staic-SA-infinite}, we propose our static
string-averaging method for solving the BAP in the infinite case,
from which the well-known Combettes algorithm follows.

\section{Previous related works \label{sec:Previous-related-works}}

An early approach, based on projection operators, is John von Neumann's
alternating projection method \cite{NeumanBook}, for solving the
BAP with two closed linear subspaces. It has been widely studied and
generalized by many authors, see, e.g., Bauschke and Borwein \cite{BauschkeBorweinAlternatingProjection},
Deutsch \cite{DeutchAlternatingProjections}, Kopecká and Reich \cite{ReichKopecka}
and Deutsch and Hundal \cite{DeutschHundal}. For the general case
of arbitrary convex sets, the Dykstra's algorithm is a suitable modification
of the alternating projections method to solve the BAP. This algorithm
was first introduced by Dykstra in \cite{Dykstra} for closed and
convex cones in finite-dimensional Euclidean spaces, and later extended
by Boyle and Dykstra in \cite{BoyleDykstra} for closed and convex
sets in a Hilbert space. Additional projections approaches can be
found in Aragón Artacho and Campoy \cite{Artacho} and Bergman, Censor,
Reich, and Zepkowitz-Malachi \cite{ReichProjectionsHyperplanes}.

For a recent bibliography of papers and monographs on projection methods
see Censor and Cegielski \cite{CensorCegielski}. A well-known method
for solving the more general BAP is Halpern's algorithm \cite{Halpern},
whose strong convergence to the solution, under various sets of assumptions
on the parameters, has been proved by several authors. The main contributions
are due to Lions, to Wittman and to Bauschke, see Bauschke's paper
\cite{HLWB}. Additional literature background of the BAP and related
problems appear in the excellent literature review of \cite{Artacho},
from which we adapted some of the above.

The literature on string-averaging algorithms has also expanded since
its first presentation in \cite{Pink-book} and further related work
has been published. In Crombez \cite{Crombez}, a string-averaging
algorithm for solving the common fixed point problem for a finite
family of paracontracting operators is proposed. Censor and Segal
in \cite{SegalCensor} suggested a string-averaging solution for the
common fixed point problem for a finite family of sparse operators.
In Censor and Tom \cite{CensorTom}, a study on the behavior of a
string-averaging algorithm for inconsistent convex feasibility problems
is done.

A generalized notion of string-averaging is the notion of dynamic
string-averaging, in which one is able to use in the iterative process
pairs $(\varOmega,w)$ from a pre-defined set, denoted by $\mathcal{M}_{*}\subset\mathcal{M}$
(see, e.g., \cite{CensorZaslavski2}). In every $(\varOmega,w)\in\mathcal{M}_{*}$,
the length of every $t\in\varOmega$ is bounded and $w(t)$ is bounded
away from zero. An example of a dynamic string-averaging algorithm
can be found in \cite{CensorZaslavski1}. That iterative process generates
a convergent sequence whose limit is a point in the intersection of
a finite family of closed convex sets, namely, a solution to the CFP.
Recently, Censor, Brooke and Gibali in \cite{CQmethods} were able
to construct a dynamic string-averaging method for solving the multiple-operator
split common fixed point problem (see, e.g., \cite[Problem 1]{CQmethods})
for families of cutters in Hilbert spaces.

\section{Preliminaries \label{sec: Preliminaries}}

Throughout our work we denote by $H$ a real Hilbert space with the
inner product $\left\langle \cdot\,,\,\cdot\right\rangle $ and induced
norm $\Vert\cdot\Vert$, and by $\mathbb{N}$ the set of all natural
numbers including zero. Let $D\subseteq H$ be a nonempty set and
let $T:D\rightarrow H$ be an operator. A vector $x\in D$ is a fixed
point of $T$ if it satisfies $T(x)=x.$ The set of all fixed points
of $T$ is denoted by $Fix(T):=\left\{ x\in D\,|\,T(x)=x\right\} .$
If $D$ is a closed convex set, then for every $u\in H$ there exists
a unique point $y\in D$ such that 
\begin{equation}
\Vert u-y\Vert=\text{inf}\left\{ \Vert u-d\Vert\,|\,d\in D\right\} .
\end{equation}
Such a point $y$ is called the projection of $u$ onto $D$ and denoted
by $P_{D}(u).$

We now recall some definitions regarding various classes of operators.
\begin{defn}
\label{def: nonexpansive-class-operators}Let $D$ be a nonempty subset
of $H$ and let $T:D\rightarrow H$. Then $T$ is:\\
(i) Nonexpansive (NE) if 
\begin{equation}
\Vert T(x)-T(y)\Vert\leq\Vert x-y\Vert,\:\forall x,y\in D.
\end{equation}
(ii) Firmly nonexpansive (FNE) if
\begin{equation}
\Vert T(x)-T(y)\Vert^{2}\leq\left\langle x-y,T(x)-T(y)\right\rangle ,\:\forall x,y\in D.
\end{equation}
(iii) Quasi-nonexpansive (QNE) if

\begin{equation}
\Vert T(x)-y\Vert\leq\Vert x-y\Vert,\:\forall x\in D,\:y\in Fix(T).
\end{equation}
(iv) Strictly quasi-nonexpansive (sQNE) if
\begin{equation}
\Vert T(x)-y\Vert<\Vert x-y\Vert,\:\forall x\in D\setminus Fix(T),\:y\in Fix(T).
\end{equation}
(v) $C$-strictly quasi-nonexpansive ($C$-sQNE) if $T$ is quasi-nonexpansive
and 
\begin{equation}
\Vert T(x)-y\Vert<\Vert x-y\Vert,\:\forall x\in D\setminus Fix(T),\:y\in C,
\end{equation}
where $C\neq\emptyset$ and $C\subseteq Fix(T)$.

A useful fact which can be found, e.g., in \cite[Proposition 2.1.11]{Cegielski},
is that the fixed point set of a nonexpansive operator is a closed
convex set. Since an intersection of any collection of closed convex
sets is a closed convex set (see, e.g., Lemma 1.13 and Example 2.3
in Deutsch \cite{DeutschBook}) it follows that, the set of common
fixed point sets of a family of nonexpansive operators of any cardinality
is a closed convex set and, thus, the projection of any given point
$u$ onto this set is well-defined, provided that the intersection
is nonempty.

All our string-averaging methods use sequences of real numbers, called
steering sequences.
\end{defn}

\begin{defn}
\label{def:Steering-Seq}(\textbf{steering sequences}). A real sequence
$(\lambda_{k})_{k\in\mathbb{N}}$ is called a \textit{steering sequence}
if it has the following properties:
\begin{equation}
\lambda_{k}\in[0,1]\enspace\text{for}\enspace\text{all}\enspace\,k\geq0,\enspace\text{and}\enspace\underset{k\rightarrow\infty}{lim}\lambda_{k}=0,\label{eq:SteeringSeq-first condition}
\end{equation}
\begin{equation}
\sum_{k=0}^{\infty}\lambda_{k}=+\infty\,\text{(\text{or,}\enspace\text{equivalently, }}\prod_{k=0}^{\infty}(1-\lambda_{k})=0),\label{eq:SteeringSeq-second condition}
\end{equation}
\begin{equation}
\sum_{k=0}^{\infty}|\lambda_{k+1}-\lambda_{k}|<\infty.\label{eq:SteeringSeq-third condition}
\end{equation}
\end{defn}

Observe that although $\lambda_{k}\in[0,1]$ the definition rules
out the option of choosing all $\lambda_{k}$ equal to zero or all
equal to one because of contradictions with the other properties.
Infinitely many zeros are possible only if the remaining nonzero elements
obey all properties. The third property in (\ref{eq:SteeringSeq-third condition})
was introduced by Wittmann, see, e.g., the recent review paper of
López, Martin-Márquez and Xu \cite{Lopez-Xu}.

Lemma \ref{lemma-1- link} below is composed of several known claims
which will be used in the sequel. We supply pointers to the proof
of the lemma for completeness.
\begin{lem}
\label{lemma-1- link} Let $D$ be a nonempty closed convex subset
of $H$ and let $T:D\rightarrow D$ be FNE. Then\\
(i) T is NE. \\
(ii) If $Fix(T)\neq\emptyset$ then T is QNE.\\
(iii) If $Fix(T)\neq\emptyset$ then T is sQNE.
\end{lem}

\begin{proof}
For (i) and (ii) see \cite[Theorem 2.2.4]{Cegielski} and \cite[Lemma 2.1.20]{Cegielski},
respectively. (iii) follows from the statement on page 70 of \cite{BauschkeCombettes}.
\end{proof}

\section{String-averaging methods for best approximation to the common fixed
points set of a family of firmly nonexpansive operators: General\label{sec: string-averaging methods - FNEs}}

In our present work we consider the best approximation problem with
respect to the common fixed points set of a family of FNEs. Our overall
aim is to develop and investigate string-averaging algorithms for
this problem. In the string-averaging algorithmic scheme, one constructs
from a given family of operators, a family of, so-called, string-operators
which are certain compositions of some of the operators from the given
family.

According to these the string-averaging algorithm proceeds in its
iterative process. We show that such string-averaging algorithmic
schemes converge to the projection of a given point (commonly called
the ``anchor'') to the common fixed point set of the given family.

We are able to ensure convergence of our string-averaging methods
by demanding that the operators of the given family be FNEs. There
are well-known links between the classes of FNEs, NEs, sQNEs and QNEs,
defined above, that help us in our analysis. To take advantage of
these links we use Corollary 4.50 and Proposition 4.47 in Bauschke
and Combettes' book \cite{BauschkeCombettes}, which are proved for
a finite family of sQNEs and a finite family of QNEs, respectively.
We extend the usage of \cite[Proposition 4.47]{BauschkeCombettes}
to the countable case in order to determine, in Section \ref{sec:The-infinite-case},
the point to which our string-averaging algorithm converges to.

\section{String-averaging methods for best approximation to the common fixed
points set of a family of firmly nonexpansive operators: The finite
case\label{sec:The-finite-case}}

In our work we develop string-averaging algorithms for two distinct
situations. One is the finite case wherein the family of given FNEs
is finite. The other is when the family of given FNEs is countably
infinite. In this section we consider the finite case. We start by
defining the terms which we use throughout this section.
\begin{defn}
\label{def:IndexVec-StringOp-FitOmega-FiniteCase}Let $D$ be a nonempty
closed convex subset of $H$, let $(T_{i})_{i=1}^{m}$ be a finite
family of self-mapping operators on $D.$ An\textit{ index vector}
is a vector of the form $t=(t_{1},t_{2},\ldots,t_{p})$ such that
$t_{\ell}\in\{1,2,\ldots,m\}$ for all $\ell\in\{1,2,\ldots,p\}$.
For a given index vector $t=(t_{1},t_{2},\ldots,t_{q})$ we denote
its length (i.e., the number of its components) by $\gamma(t)=q$,
and define the operator $T[t]$ as the composition of the operators
$T_{i}$ whose indices appear in the index vector $t$, namely,
\begin{equation}
T[t]:=T_{t_{q}}T_{t_{q-1}}\cdots T_{t_{1}},\label{eq:string}
\end{equation}
 and call it a \textit{string operator}. A finite set $\varOmega$
of index vectors is called \textit{fit} if for each $i\in\left\{ 1,2,\ldots,m\right\} $,
there exists a vector $t=(t_{1},t_{2},\ldots,t_{p})\in\varOmega$
such that $t_{\ell}=i$ for some $\ell\in\left\{ 1,2,\ldots,p\right\} $.
As in \cite{CensorZaslavski1}, we denote by $\mathcal{M}$ the collection
of all pairs $(\varOmega,w)$, where $\Omega$ is a fit finite set
of index vectors and $w:\Omega\rightarrow(0,1]$ is such that $\sum_{t\in\varOmega}w(t)=1$.
\end{defn}

As mentioned above, \cite[Corollary 4.50]{BauschkeCombettes} and
\cite[Proposition 4.47]{BauschkeCombettes} are cornerstones to our
proofs of convergence. The following are slight rephrased versions
of them, respectively, adapted to our notations and needs.
\begin{prop}
\label{prop-1-Corollary 4.50 BC} Let $D$ be a nonempty subset of
$H$ and let $(T_{i})_{i=1}^{m}$ be a finite family of self sQNEs
on $D$ such that $\cap_{i=1}^{m}Fix(T_{i})\neq\emptyset$ and set
$T=T_{1}T_{2}\cdots T_{m}$. Then $T$ is sQNE and $Fix(T)=\cap_{i=1}^{m}Fix(T_{i})$.
\end{prop}

\begin{prop}
\label{prop-2- Prop4.47}Let $D$ be a nonempty subset of $H$, let
$(T_{i})_{i=1}^{m}$ be a finite family of self QNEs on $D$ such
that $\cap_{i=1}^{m}Fix(T_{i})\neq\emptyset$ and let $(w_{i})_{i=1}^{m}$
be a sequence of strictly positive real numbers such that $\sum_{i=1}^{m}w_{i}=1$.
Then $Fix(\sum_{i=1}^{m}w_{i}T_{i})=\cap_{i=1}^{m}Fix(T_{i})$.
\end{prop}

The following lemma combines all of the above into a useful auxiliary
tool which is used repeatedly. We consider the operator $T:=\sum_{t\in\Omega}w(t)T[t]$
which is called a \textit{string-averaging operator.}
\begin{lem}
\label{lemma-2 - correct set} Let $D$ be a nonempty closed convex
subset of $H$ and let $(T_{i})_{i=1}^{m}$ be a finite family of
self FNEs on $D$ such that $F:=\cap_{i=1}^{m}Fix(T_{i})\neq\emptyset$.
Let $(\varOmega,w)\in\mathcal{\mathcal{M}}$ and let $T=\sum_{t\in\Omega}w(t)T[t]$.
Then $Fix(T)=\cap_{i=1}^{m}Fix(T_{i})$.
\end{lem}

\begin{proof}
Note that $F\neq\emptyset$ and (\ref{eq:string}) imply that for
every $t\in\varOmega$ and for every $x\in F$, $T[t](x)=x$ (since
$T[t]$ is a composition). Hence,
\begin{equation}
F\subseteq Fix(T[t]).\label{eq:Lemma2-1}
\end{equation}
Therefore,
\begin{equation}
F\subseteq\cap_{t\in\Omega}Fix(T[t]).\label{eq:Lemma2-2}
\end{equation}
We want to use Proposition \ref{prop-2- Prop4.47} for the family
$(T[t])_{t\in\varOmega}$. Since Every FNE is NE (see Lemma \ref{lemma-1- link}(i)),
for every $t\in\varOmega$ we apply \cite[Lemma 2.1.12(ii)]{Cegielski}
to the family $(T_{t_{\ell}})_{\ell=1}^{q}$ and conclude that the
string operator $T[t]=T_{t_{q}}T_{t_{q-1}}\cdots T_{t_{1}}$ is NE.
Thus, from (\ref{eq:Lemma2-1}) and due to the fact that every NE
with a fixed point is QNE, the family of operators $(T[t])_{t\in\varOmega}$
is a finite family of QNEs. This, together with (\ref{eq:Lemma2-2}),
yields, according to Proposition \ref{prop-2- Prop4.47}, that
\begin{equation}
Fix(T)=\cap_{t\in\varOmega}Fix(T[t]).\label{eq:15}
\end{equation}
 $F\neq\emptyset$ implies that $Fix(T_{i})\neq\emptyset$, thus,
by Lemma \ref{lemma-1- link}(iii), $(T_{i})_{i=1}^{m}$ is a family
of sQNEs and, so, for every $t\in\varOmega$, applying Proposition
\ref{prop-1-Corollary 4.50 BC} with the family $(T_{t_{\ell}})_{\ell=1}^{q}$
yields that the string operator $T[t]$ satisfies 
\begin{equation}
Fix(T[t])=\cap_{\ell=1}^{q}Fix(T_{t_{\ell}}).\label{eq:Lemma2-3}
\end{equation}
From the fitness of $\varOmega$ and from (\ref{eq:Lemma2-3}), we
get
\begin{equation}
\cap_{t\in\varOmega}Fix(T[t])=\cap_{t\in\varOmega}(\cap_{\ell=1}^{q}Fix(T_{t_{\ell}}))=\cap_{i=1}^{m}Fix(T_{i}),
\end{equation}
and, overall, from (\ref{eq:15})
\begin{equation}
Fix(T)=\cap_{i=1}^{m}Fix(T_{i}).
\end{equation}
\end{proof}
We are now ready to propose the new string-averaging methods for solving
the best approximation problem and prove their convergence.

\subsection{The static string-averaging method for the finite case \label{subsec:static-SA-finite}}

First we discuss string-averaging methods in which a single pair $(\varOmega,w)\in\mathcal{\mathcal{\mathcal{M}}}$
is picked at the outset and kept fixed throughout the iterative process.
Such string-averaging methods will be termed ``static string-averaging
methods''. We will make use of the convergence theorem in \cite{Halpern}.
Halpern's algorithm is a sequential algorithm which generates a sequence
via the iterative process 
\begin{equation}
x^{k+1}=\lambda_{k}u+(1-\lambda_{k})S(x^{k}),\label{eq: Halpern}
\end{equation}
where $S:D\rightarrow D$ is NE with $Fix(S)\neq\emptyset,$ the anchor
point $u\in D$ is given and fixed, the initialization $x^{0}\in D$
is an arbitrary point, and the sequence $(\lambda_{k})_{k\in\mathbb{N}}$
is a steering sequence as in Definition \ref{def:Steering-Seq}. Its
proof of convergence can be found also in \cite[Theorem 30.1]{BauschkeCombettes}.
We present a slightly rephrased version of this theorem without a
proof.
\begin{thm}
\label{thm-1-Halpern} Let $D$ be a nonempty closed convex subset
of $H$ and let $S$ be a self NE on $D$ such that $Fix(S)\neq\emptyset$.
Let $(\lambda_{k})_{k\in\mathbb{N}}$ be a steering sequence and let
$u,x^{0}\in D$. Then any sequence generated by (\ref{eq: Halpern})
converges strongly to $P_{Fix(S)}(u)$.
\end{thm}

Our new static string-averaging algorithm for a finite family of FNEs
is as follows.\medskip{}
\textbf{}\\
\textbf{Algorithm 1. The static string-averaging algorithm for solving
the best approximation problem in the finite case.}\\
\textbf{Initialization:} Choose a single pair $(\varOmega,w)\in\mathcal{M}$
and an arbitrary $x^{0}\in D.$\textbf{}\\
\textbf{Iterative step:} Given the current iterate $x^{k},$ calculate
the next iterate $x^{k+1}$ by
\begin{equation}
x^{k+1}=\lambda_{k}u+(1-\lambda_{k})\sum_{t\in\Omega}w(t)T\left[t\right]\left(x^{k}\right),
\end{equation}
where $(\lambda_{k})_{k\in\mathbb{N}}$ is a steering sequence, $u$
is the given anchor point and $w(t)$ and $T[t]$ are as in Definition
\ref{def:IndexVec-StringOp-FitOmega-FiniteCase}.\medskip{}

The convergence proof of Algorithm 1 follows.
\begin{thm}
\label{thm-2-FiniteCase-SA}Let $D$ be a nonempty closed convex subset
of $H$ and let $(T_{i})_{i=1}^{m}$ be a finite family of self FNEs
on $D$ such that $F:=\cap_{i=1}^{m}Fix(T_{i})\neq\emptyset$. Let
$(\varOmega,w)\in\mathcal{\mathcal{\mathcal{M}}}$ be fixed, let $(\lambda_{k})_{k\in\mathbb{N}}$
be a steering sequence and let $u,x^{0}\in D$. Then any sequence
$(x^{k})_{k\in\mathbb{N}}$, generated by Algorithm 1, converges strongly
to $P_{F}(u)$.
\end{thm}

\begin{proof}
For $(T_{i})_{i=1}^{m}$ and $(\varOmega,w)\in\mathcal{\mathcal{M}}$
consider the family of operators $(T[t])_{t\in\varOmega}$ and define
the string-averaging operator $T:=\sum_{t\in\varOmega}w(t)T\left[t\right]$.
We show first that the operator $T$ is NE and that $Fix(T)=\cap_{i=1}^{m}Fix(T_{i})$.
From the proof of Lemma \ref{lemma-1- link} we conclude that $T$
is a convex combination of NEs and, thus, since a convex combination
of NEs is a NE (see, e.g., \cite[Lemma 2.1.12(i)]{Cegielski}), $T$
is NE. Moreover, $Fix(T)$ is not empty since it contains $F.$ Applying
Halpern's Theorem \ref{thm-1-Halpern} with $T$ in the role of $S$,
any sequence $(x^{k})_{k\in\mathbb{N}}$, generated by Algorithm 1,
converges strongly to $P_{Fix(T)}(u)$. Now, applying Lemma \ref{lemma-2 - correct set}
on $(T_{i})_{i=1}^{m}$ together with $(\varOmega,w)$, results in
\begin{equation}
Fix(T)=\cap_{i=1}^{m}Fix(T_{i}),\label{eq: Theorem-3-1}
\end{equation}
and, therefore, 
\begin{equation}
x^{k}\rightarrow P_{F}(u).
\end{equation}
\end{proof}
This concludes our treatment of the static string-averaging algorithm
for solving the best approximation problem in the finite case. We
make the following remark about it.
\begin{rem}
\label{Remark-1}Theorem \ref{thm-2-FiniteCase-SA} is related to
two important results that appear in \cite{BauschkeCombettes}. It
generalizes Corollaries 30.2 and 30.3 in that book, from the algorithmic
structural point of view, because the algorithms there are \textit{fully-simultaneous}
and \textit{fully-sequential}, respectively. These two algorithmic
options are special cases of the static string-averaging algorithm
that are obtained by choosing either strings of length one with every
index $i=1,2,\ldots,m$ appearing exactly in one string or by choosing
to use a single string that includes all indices $i=1,2,\ldots,m,$
respectively. However, our Theorem \ref{thm-2-FiniteCase-SA} cannot
be considered as a generalization of those corollaries because the
corollaries deal with NEs while we restrict our analysis of the string-averaging
algorithmic structure to only FNEs. The question whether or not our
Theorem \ref{thm-2-FiniteCase-SA} can be proven for NEs remains open.
\end{rem}

Next we expand our results to a non-static case.

\subsection{The quasi-dynamic string-averaging method for the finite case\label{subsec:The-quasi-dynamic}}

The key adjustment which we did in the construction of Algorithm 1
in order to prove its convergence with the aid of Halpern's algorithm,
was the repeated use of the same single fixed pair $(\varOmega,w)$
in all iterations $k\geq1$. This reuse of a fixed $(\varOmega,w)$
throughout the whole iterative process of Algorithm 1 is a special
case of a more general method, mentioned briefly already in Section
\ref{sec:Previous-related-works}, called the \textit{dynamic string-averaging
method}. The dynamic string-averaging algorithmic scheme allows to
pick and use in every step of the iterative process any pair $(\varOmega,w)$
from a pre-defined set $\mathcal{M}_{*}\subset\mathcal{M}$. The set
$\mathcal{M}_{*}$ was defined in \cite[Equation (21)]{CensorZaslavski1}
as follows.
\begin{defn}
{\large{}\label{def: M*}}Fix a number $\Delta\in(0,\,1/m)$ and an
integer $\bar{q}\geq m$ and denote by $\mathcal{M}_{*}\equiv\mathcal{M}_{*}(\Delta,\bar{q})$
the set of all $(\varOmega,w)\in\mathcal{M}$ such that the lengths
of the strings are bounded and the weights are bounded away from zero,
namely,
\begin{equation}
\mathcal{M}_{*}:=\,\left\{ (\varOmega,w)\in\mathcal{M}\,|\,\gamma(t)\leq\bar{q}\:\text{and}\,\Delta\leq w(t),\,\text{for}\,\text{all}\,t\in\varOmega\right\} .\label{eq:M*}
\end{equation}
\end{defn}

The set $\mathcal{M}_{*}$ is an infinite subset of $\mathcal{M}$
but in our quasi-dynamic string-averaging method, proposed below,
we must confine ourselves to a finite subset of $\mathcal{\mathcal{M}}.$
So, instead of choosing the pairs $(\varOmega,w)$ from $\mathcal{M}_{*}$
we choose them from a finite-cardinality subset, denoted by $\mathcal{M}'$$\subset\mathcal{\mathcal{\mathcal{M}}}$,
and use them in a cyclic manner. The finiteness of $\mathcal{M}'$
guarantees that it is actually a subset of $\mathcal{M}_{*}$. As
one can tell, such an algorithm is indeed not as ``dynamic'' as
Algorithm 6 of \cite{CensorZaslavski1} or Algorithm 3 of \cite{CensorZaslavski2}
but, nevertheless, it is not static as Algorithm 1. Therefore, we
name it \textit{quasi-dynamic string-averaging method}.

Next, we explain how the construction of our quasi-dynamic string-averaging
algorithm is done. Let us construct a sequence that is an ordered
version of $\mathcal{M}'$ as follows: Let $\sigma:\mathcal{M}'\rightarrow\left\{ 1,2,\ldots\,,\left|\mathcal{M}'\right|\right\} $
be a one-to-one correspondence and denote, for every $r\in\left\{ 1,2,\ldots\,,\left|\mathcal{M}'\right|\right\} $,
$S_{r}:=(\varOmega_{r},w_{r})$. Let $\left(S_{r}\right)_{r=1}^{\left|\mathcal{M}'\right|}$
be a sequence of all $(\varOmega,w)\in\mathcal{M}'$, sorted by $\sigma((\varOmega,w))$.
Namely, $S_{1}$ is the pair $(\varOmega,w)$ such that $\sigma((\varOmega,w))=1$,
$S_{2}$ is the pair $(\varOmega,w)$ such that $\sigma((\varOmega,w))=2$
and so forth until $S_{\left|\mathcal{M}'\right|}$ is the pair $(\varOmega,w)$
such that $\sigma((\varOmega,w))=\left|\mathcal{M}'\right|$.

The following is our quasi-dynamic string-averaging algorithm.\medskip{}

\begin{flushleft}
\textbf{Algorithm 2. The quasi-dynamic string-averaging algorithm
for solving the best approximation problem in the finite case.}\\
\textbf{Initialization:} $x^{0}\in D$\textbf{ }is arbitrary.\textbf{}\\
\textbf{Iterative step:} Given the current iterate $x^{k},$ calculate
the next iterate $x^{k+1}$ by
\begin{equation}
x^{k+1}=\lambda_{k}u+(1-\lambda_{k})\sum_{t\in\Omega_{j(k)}}w_{j(k)}(t)T[t](x^{k}),\label{eq: Algorithm 2 scheme}
\end{equation}
where $(\lambda_{k})_{k\in\mathbb{N}}$ is a steering sequence, $u$
is the given anchor point, $j(k)=k\,mod\left|\mathcal{M}'\right|+1$
for all $k\geq0,$ is a cyclic control sequence (see, e.g., \cite[Definition 4]{SegalCensor})
and $\varOmega_{j(k)}$ and $w_{j(k)}$ are the elements of the pair
$S_{j(k)}=(\varOmega_{j(k)},w_{j(k)})$, respectively.\medskip{}
\par\end{flushleft}

To prove the convergence of Algorithm 2 we make use of the following
theorem, which is a slightly rephrased version of Theorem 3.1 in \cite{HLWB}.
\begin{thm}
\label{thm-3-HLWB}Let $D$ be a nonempty closed convex subset of
$H$ and let $(T_{i})_{i=1}^{m}$ be a finite family of self NEs on
$D$ such that $F:=\cap_{i=1}^{m}Fix(T_{i})\neq\emptyset.$ Assume
that
\begin{equation}
\begin{aligned}F & =Fix(T_{m}T_{m-1}\cdots T_{1})=Fix(T_{1}T_{m}\cdots T_{3}T_{2})\\
 & =\ldots=Fix(T_{m-1}T_{m-2}\cdots T_{1}T_{m}).
\end{aligned}
\label{eq: Theorem-4-1}
\end{equation}
Let $i(k)=k\,mod\,m+1$ be a cyclic control sequence, let $(\lambda_{k})_{k\in\mathbb{N}}$
be a steering sequence and let $u,x^{0}\in D$. Then any sequence
$(x^{k})_{k\in\mathbb{N}}$ generated by 
\begin{equation}
x^{k+1}=\lambda_{k}u+(1-\lambda_{k})T_{i(k)}(x^{k}),\label{eq:Theorem-4-2}
\end{equation}

converges strongly to $P_{F}(u)$.
\end{thm}

With the aid of the sequence $(S_{r})_{r=1}^{\left|\mathcal{M}'\right|}$
and Theorem \ref{thm-3-HLWB} we present in the next theorem a proof
of convergence of Algorithm 2.
\begin{thm}
\label{thm-4- QDSA-HLWB}Let $D$ be a nonempty closed convex subset
of H and let $(T_{i})_{i=1}^{m}$ be a finite family of self FNEs
on $D$ such that $F:=\cap_{i=1}^{m}Fix(T_{i})\neq\emptyset$. Let
$\mathcal{M}'$ be a finite subset of $\mathcal{\mathcal{\mathcal{M}}},$
let $j(k)=k\,mod\left|\mathcal{M}'\right|+1$ , for all $k\geq0,$
be a cyclic control sequence and let $(\lambda_{k})_{k\in\mathbb{N}}$
be a steering sequence. Let $u,x^{0}\in D.$ Then any sequence $(x^{k})_{k\in\mathbb{N}}$,
generated by Algorithm 2, converges strongly to $P_{F}(u)$.
\end{thm}

\begin{proof}
Define the finite family of operators $(T_{S_{r}})_{r=1}^{\left|\mathcal{M}'\right|}$
by $T_{S_{r}}:=\sum_{t\in\varOmega_{r}}w_{r}(t)T[t]$ where for every
$r\in\left\{ 1,2,\ldots\,,\left|\mathcal{M}'\right|\right\} $, $S_{r}:=(\varOmega_{r},w_{r})$.
We first show that for every $r\in\left\{ 1,2,\ldots\,,\left|\mathcal{M}'\right|\right\} $
the operator $T_{S_{r}}$ is both NE and sQNE. By similar arguments
to those made for $T$ in the proof of Theorem \ref{thm-2-FiniteCase-SA},
it follows that $(T_{S_{r}})_{r=1}^{\left|\mathcal{M}'\right|}$ is
a family of NEs. From the proof of Lemma \ref{lemma-2 - correct set}
we deduce that $(T[t])_{t\in\varOmega_{r}}$ is a family of QNEs such
that
\begin{equation}
\emptyset\neq F\subseteq\cap_{t\in\varOmega_{r}}Fix(T[t]),\label{eq:Theorem-5-1-1}
\end{equation}
 for every $r\in\left\{ 1,2,\ldots\,,\left|\mathcal{M}'\right|\right\} $.
Therefore, we are able to apply Proposition \ref{prop-2- Prop4.47}
to the family $(T[t])_{t\in\varOmega_{r}}$ (in place of the family
$(T_{i})_{i=1}^{m}$) for every $r\in\left\{ 1,2,\ldots\,,\left|\mathcal{M}'\right|\right\} $,
which results in
\begin{equation}
Fix(T_{S_{r}})=\cap_{t\in\varOmega_{r}}Fix(T[t]).\label{eq:Theorem-5-1}
\end{equation}
Therefore, by (\ref{eq:Theorem-5-1-1}),
\begin{equation}
Fix(T_{S_{r}})\neq\emptyset.\label{eq:Theorem-5-1-2}
\end{equation}
Next we show that $T_{S_{r}}$ is sQNE. According to Lemma \ref{lemma-1- link}(iii),
each $T_{i}$ is sQNE and, hence, similarly to the analysis made in
the proof of Lemma \ref{lemma-2 - correct set}, with the aid of \cite[Corollary 4.50]{BauschkeCombettes}
instead of \cite[Lemma 2.1.12(ii)]{Cegielski}, $T[t]$ is sQNE for
every $t\in\varOmega_{r}$, and for every $r\in\left\{ 1,2,\ldots\,,\left|\mathcal{M}'\right|\right\} $.
In particular, the latter shows that for a given $r$ and any $t\in\varOmega_{r}$,
we have, due to (\ref{eq:Theorem-5-1}), that for all $x\in D\setminus Fix(T[t])$
and $y\in Fix(T_{S_{r}})$, 
\begin{equation}
\Vert T[t](x)-y\Vert<\Vert x-y\Vert.
\end{equation}
Thus, $T[t]$ is $C$-sQNE with $C:=Fix(T_{S_{r}}).$ Consequently,
from (\ref{eq:Theorem-5-1}), (\ref{eq:Theorem-5-1-2}) and \cite[Theorem 2.1.26(i)]{Cegielski},
$T_{S_{r}}$is $C$-sQNE for every $r\in\left\{ 1,2,\ldots\,,\left|\mathcal{M}'\right|\right\} $
and, so, by the definition of $C$ and \cite[page 47]{Cegielski},
$T_{S_{r}}$ is sQNE for every $r\in\left\{ 1,2,\ldots\,,\left|\mathcal{M}'\right|\right\} $.

Therefore, since $\emptyset\neq F\subseteq\cap_{r=1}^{\left|\mathcal{M}'\right|}Fix(T_{S_{r}})$,
we can apply Proposition \ref{prop-1-Corollary 4.50 BC} to the family
$(T_{S_{r}})_{r=1}^{\left|\mathcal{M}'\right|}$, which yields that
(\ref{eq: Theorem-4-1}) holds for $(T_{S_{r}})_{r=1}^{\left|\mathcal{M}'\right|}$.
Hence, since the family $(T_{S_{r}})_{r=1}^{\left|\mathcal{M}'\right|}$
is a family of NEs and it satisfies (\ref{eq: Theorem-4-1}), we let
the role of the family $(T_{i})_{i=1}^{m}$ in Theorem \ref{thm-3-HLWB}
to be played by the family $(T_{S_{r}})_{r=1}^{\left|\mathcal{M}'\right|}$.
Moreover, by taking the sequence $(i(k))_{k\in\mathbb{N}}$ in Theorem
\ref{thm-3-HLWB} to be $(j(k))_{k\in\mathbb{N}}$, (\ref{eq: Algorithm 2 scheme})
turns out to be a special case of (\ref{eq:Theorem-4-2}) and we deduce
that any sequence $(x^{k})_{k\in\mathbb{N}}$, generated by (\ref{eq: Algorithm 2 scheme}),
converges strongly to $P_{\cap_{r=1}^{\left|\mathcal{M}'\right|}Fix(T_{S_{r}})}(u)$.

Now, (\ref{eq:Theorem-5-1}), (\ref{eq:Lemma2-3}) and the fitness
of $\varOmega$ (recall Definition \ref{def:IndexVec-StringOp-FitOmega-FiniteCase})
imply that
\begin{equation}
\cap_{r=1}^{\left|\mathcal{M}'\right|}Fix(T_{S_{r}})=\cap_{r=1}^{\left|\mathcal{M}'\right|}(\cap_{t\in\varOmega_{r}}Fix(T[t]))=\cap_{r=1}^{\left|\mathcal{M}'\right|}(\cap_{i=1}^{m}Fix(T_{i}))=F,\label{eq:intersection-DSA-HLWB-1}
\end{equation}
and, in conclusion,
\begin{equation}
x^{k}\rightarrow P_{F}(u).\label{eq:first-convergence-DSA-HLWB}
\end{equation}
\end{proof}
This concludes our treatment of the quasi-dynamic string-averaging
algorithm for solving the best approximation problem in the finite
case.

\subsection{Simultaneous string-averaging methods\label{subsec: The simultaneous QDSA}}

One possibility to define a simultaneous string-averaging method was
discussed in Remark \ref{Remark-1} above and termed ``fully simultaneous''.
By using a family of string-averaging operators $(T_{S_{r}})_{r=1}^{\left|\mathcal{M}'\right|}$,
as in Subsection \ref{subsec:The-quasi-dynamic}, and employing an
additional weight sequence $(\hat{w_{r}})_{r=1}^{\left|\mathcal{M}'\right|}$
of strictly positive real numbers such that $\sum_{r=1}^{\left|\mathcal{M}'\right|}\hat{w_{r}}=1$,
we can construct yet another algorithm of a simultaneous nature for
solving the best approximation problem to common fixed point sets
of operators in the finite case. This algorithm convexly-combines
via $(\hat{w_{r}})_{r=1}^{\left|\mathcal{M}'\right|}$ the points
$T_{S_{r}}(x^{k})$ which amounts to string-averaging the end-points
of the string operators $T[t]$ for every $t\in\varOmega_{r}$, and
for every $r\in\left\{ 1,2,\ldots,\left|\mathcal{M}'\right|\right\} $.
The scheme is as follows.

\medskip{}
\textbf{Algorithm 3.\label{alg:Str-avg + Simoultenous - FiniteCase}
The simultaneous string-averaging algorithm for solving the best approximation
problem in the finite case.}\\
\textbf{Initialization:} $x^{0}\in D.$\textbf{}\\
\textbf{Iterative step:} Given the current iterate $x^{k},$ calculate
the next iterate $x^{k+1}$ by
\begin{equation}
x^{k+1}=\lambda_{k}u+(1-\lambda_{k})\sum_{r=1}^{\left|\mathcal{M}'\right|}\hat{w_{r}}T_{S_{r}}(x^{k})=\lambda_{k}u+(1-\lambda_{k})\sum_{r=1}^{\left|\mathcal{M}'\right|}\hat{w_{r}}(\sum_{t\in\varOmega_{r}}w_{r}(t)T[t](x^{k})),
\end{equation}
where $(\lambda_{k})_{k\in\mathbb{N}}$ is a steering sequence, $u$
is the given anchor point and $(\hat{w_{r}})_{r=1}^{\left|\mathcal{M}'\right|}$
are user-chosen strictly positive real numbers such that $\sum_{r=1}^{\left|\mathcal{M}'\right|}\hat{w_{r}}=1$.\medskip{}

It is possible to obtain the convergence of Algorithm 3 from Theorem
\ref{thm-2-FiniteCase-SA} about our static string-averaging Algorithm
1. But this would limit the scope to a family $(T_{S_{r}})_{r=1}^{\left|\mathcal{M}'\right|}$
of FNEs only. Therefore, we derive the convergence of Algorithm 3
from Corollary 30.2 in \cite{BauschkeCombettes} which holds for NEs.
We do this next. Let $D$ be a nonempty closed convex subset of $H$
and let $(T_{i})_{i=1}^{m}$ be a finite family of self NEs on $D.$\medskip{}

\begin{flushleft}
\textbf{Algorithm 4.\label{alg: Simoultenous - halpern} The fully-simultaneous
algorithm for the best approximation problem in Corollary 30.2 of
\cite{BauschkeCombettes}.}\\
\textbf{Initialization:} $x^{0}\in D.$\textbf{}\\
\textbf{Iterative step:} Given the current iterate $x^{k},$ calculate
the next iterate $x^{k+1}$ by
\begin{equation}
x^{k+1}=\lambda_{k}u+(1-\lambda_{k})\sum_{i=1}^{m}w_{i}T_{i}(x^{k}),
\end{equation}
where $(\lambda_{k})_{k\in\mathbb{N}}$ is a steering sequence, $u$
is the given anchor point and $(w_{i})_{i=1}^{m}$ is a sequence of
user-chosen strictly positive numbers such that $\sum_{i=1}^{m}w_{i}=1.$
\par\end{flushleft}

\begin{flushleft}
\medskip{}
\par\end{flushleft}

A slightly rephrased version of Corollary 30.2 of \cite{BauschkeCombettes}
is as follows.
\begin{thm}
\label{thm-5-simultaneous halpern}Let $D$ be a nonempty closed convex
subset of $H$ and let $(T_{i})_{i=1}^{m}$ be a finite family of
self NEs on $D$ such that $F:=\cap_{i=1}^{m}Fix(T_{i})\neq\emptyset$.
Let $(w_{i})_{i=1}^{m}$ be a sequence of strictly positive real numbers
such that $\sum_{i\in I}w_{i}=1$ and let $(\lambda_{k})_{k\in\mathbb{N}}$
be a steering sequence. Let $u,x^{0}\in D$. Then any sequence $(x^{k})_{k\in\mathbb{N}}$,
generated by Algorithm 4, converges strongly to $P_{F}(u)$.
\end{thm}

Algorithm 3 is now a special case of Algorithm 4 and its convergence
follows from the above theorem.

\begin{thm}
\label{thm-6- simultaneous-SA}Let $D$ be a nonempty closed convex
subset of $H$ and let $(T_{i})_{i=1}^{m}$ be a finite family of
self FNEs on $D$ such that $F:=\cap_{i=1}^{m}Fix(T_{i})\neq\emptyset$.
Let $(T_{S_{r}})_{r=1}^{\left|\mathcal{M}'\right|}$ be as in Theorem
\ref{thm-4- QDSA-HLWB}, let $(\hat{w_{r}})_{r=1}^{\left|\mathcal{M}'\right|}$
be a sequence of strictly positive real numbers such that $\sum_{r=1}^{\left|\mathcal{M}'\right|}\hat{w_{r}}=1$,
and let $(\lambda_{k})_{k\in\mathbb{N}}$ be a steering sequence.
Let $u,x^{0}\in D.$ Then any sequence $(x^{k})_{k\in\mathbb{N}}$,
generated by Algorithm 3, converges strongly to $P_{F}(u)$.
\end{thm}

\begin{proof}
In the proof of Theorem \ref{thm-4- QDSA-HLWB} we showed that the
family $(T_{S_{r}})_{r=1}^{\left|\mathcal{M}'\right|}$ is a family
of NEs with a nonempty common fixed points set and, so, Theorem \ref{thm-5-simultaneous halpern}
implies that any sequence $(x^{k})_{k\in\mathbb{N}},$ generated by
Algorithm 3, converges strongly to $P_{\cap_{r=1}^{\left|\mathcal{M}'\right|}Fix(T_{S_{r}})}(u)$.
Now, from (\ref{eq:intersection-DSA-HLWB-1}), we conclude that
\begin{equation}
x^{k}\rightarrow P_{F}(u).
\end{equation}
\end{proof}

\section{String-averaging with orthogonal projections for the best approximation
to a finite family of closed convex sets\label{sec:Special case - orthogonal projctions}}

In this section we specialize our results, on string-averaging methods
for solving the best approximation problem, to orthogonal projections
since they are known to be FNEs (see, e.g., \cite[Facts 1.5(i)]{BauschkeBorwein}).
We start by looking at the static string-averaging method with orthogonal
projections, which is a special case of our Algorithm 1. We show that
the simultaneous version of the Halpern-Lions-Wittman-Bauschke (HLWB)
algorithm, used in \cite[Algorithm 5]{Censor}, and the sequential
Halpern-Wittman algorithm (see, e.g., Bauschke and Koch \cite[Algorithm 4.1]{BauschkeKoch})
are special cases of the string-averaging methods.

\subsection{The static string-averaging method for a finite family of closed
convex sets}

We write down formally the static string-averaging method for a finite
family of orthogonal projections in order to make sure that our string
operators in such a case remain well-defined. Let $D\subseteq H$
be a subset of $H$, let $(C_{i})_{i=1}^{m}$ be a finite family of
closed convex sets $C_{i}\subseteq D$ for every $i\in\left\{ 1,2,\ldots.m\right\} $.
We denote by $P_{C_{i}}:D\rightarrow C_{i}$ the orthogonal projection
onto the closed convex set $C_{i}$. We now present the static string-averaging
method for a finite family of closed convex sets.

\medskip{}

\begin{flushleft}
\textbf{Algorithm 5. \label{alg:SA-static-orthogonal projections}
The static string-averaging algorithm for the best approximation to
a finite family of closed convex sets.}\\
\textbf{Initialization: }Choose a single pair $(\varOmega,w)\in\mathcal{M}$
and an arbitrary $x^{0}\in D.$\textbf{}\\
\textbf{Iterative step:} Given the current iterate $x^{k},$ calculate
the next iterate $x^{k+1}$ by
\begin{equation}
x^{k+1}=\lambda_{k}u+(1-\lambda_{k})\sum_{t\in\Omega}w(t)P[t](x^{k}),\label{eq: SA-static-orthogonal projections}
\end{equation}
where $(\lambda_{k})_{k\in\mathbb{N}}$ is a steering sequence, $u$
is the given anchor point and $w(t)$ and $P[t]$ are as in Definition
\ref{def:IndexVec-StringOp-FitOmega-FiniteCase} with $(T_{i})_{i=1}^{m}=(P_{C_{i}})_{i=1}^{m}$
and $T[t]=P[t]$.\medskip{}
\par\end{flushleft}

The proof of convergence of Algorithm 5 follows.
\begin{thm}
\label{thm-10-SA-static-orthogonal projections} Let $D$ be a nonempty
closed convex subset of $H$, let $(C_{i})_{i=1}^{m}$ be a finite
family of closed convex sets $C_{i}\subseteq D$ for every $i\in\left\{ 1,2,\ldots.m\right\} $
such that $C:=\cap_{i=1}^{m}C_{i}\neq\emptyset$. Let $(\varOmega,w)\in\mathcal{M}$,
let $(\lambda_{k})_{k\in\mathbb{N}}$ be a steering sequence and let
$u,x^{0}\in D$. Then any sequence $(x^{k})_{k\in\mathbb{N}}$, generated
by Algorithm 5, converges strongly to $P_{C}(u)$.
\end{thm}

\begin{proof}
This is a straightforward consequence of Theorem \ref{thm-2-FiniteCase-SA}.
Set the finite family of FNEs in Theorem \ref{thm-2-FiniteCase-SA}
to be $(P_{C_{i}})_{i=1}^{m}$. Then (\ref{eq: SA-static-orthogonal projections})
turns out to be special case of (\ref{eq: Halpern}) and, so, by Theorem
\ref{thm-2-FiniteCase-SA}, any sequence $(x^{k})_{k\in\mathbb{N}}$,
generated by Algorithm 5, converges strongly to $P_{C}(u).$
\end{proof}

\subsection{The simultaneous Halpern-Lions-Wittman-Bauschke algorithm as a special
case}

In \cite{Combettes} a parallel version of Halpern's algorithm leads
to a simultaneous Halpern-Lions-Wittman-Bauschke (HLWB) algorithm
for a countable family of FNEs, see also Deutsch and Yamada \cite{YamadaDeutsch}.
Here we present a simultaneous HLWB algorithm for the case of a finite
family of closed convex sets. The convergence of this algorithm follows
directly by choosing strings that are singletons such that each index
$i\in\left\{ 1,2,\ldots.m\right\} $ appears in one string. Thus,
this algorithm is not only a consequence of the above mentioned work
of Combettes but also a consequence of our work here. Again, let $D$
be a nonempty closed convex subset of $H$, let $(C_{i})_{i=1}^{m}$
be a finite family of closed convex sets $C_{i}\subseteq D$ for every
$i\in\left\{ 1,2,\ldots.m\right\} $ such that $C:=\cap_{i=1}^{m}C_{i}\neq\emptyset$.\medskip{}

\begin{flushleft}
\textbf{Algorithm 6. \label{alg:simultaneous Halpern-Lions-Wittman-Bauschke  }
The simultaneous HLWB algorithm with orthogonal projections for the
best approximation to a finite family of closed convex sets.}\\
\textbf{Initialization:} $x^{0}=u.$\textbf{}\\
\textbf{Iterative step:} Given the current iterate $x^{k},$ calculate
the next iterate $x^{k+1}$ by
\begin{equation}
x^{k+1}=\lambda_{k}u+(1-\lambda_{k})\sum_{i=1}^{m}w_{i}P_{i}(x^{k}),\label{eq: s-Halpern=002013Lions=002013Wittmann=002013Bauschke}
\end{equation}
where $(\lambda_{k})_{k\in\mathbb{N}}$ is a steering sequence, $u$
is the given anchor point and $(w_{i})_{i=1}^{m}$ is a sequence of
user-chosen strictly positive real numbers such that $\sum_{i\in I}w_{i}=1$.
\par\end{flushleft}

\subsection{The Halpern-Wittman algorithm as a special case}

Following the results in \cite{Halpern}, Wittman \cite{Wittman}
showed the convergence of an algorithm that is presented in \cite[Algorithm 4.1]{BauschkeKoch}
and is named there the Halpern-Wittman algorithm. It is designed for
a finite family of orthogonal projections and a specific steering
sequence. This algorithm is presented below as Algorithm 7 and its
convergence follows directly from the convergence of Algorithm 5 by
putting all indices of $i\in\left\{ 1,2,\ldots.m\right\} $ into a
single string. Again, let $D$ be a nonempty closed convex subset
of $H$, let $(C_{i})_{i=1}^{m}$ be a finite family of closed convex
sets $C_{i}\subseteq D$ for every $i\in\left\{ 1,2,\ldots.m\right\} $
such that $C:=\cap_{i=1}^{m}C_{i}\neq\emptyset$.

\medskip{}

\begin{flushleft}
\textbf{Algorithm 7. \label{alg:Halpern-Wittman algorithm} The Halpern-Wittman
algorithm.}\\
\textbf{Initialization:} $x^{0}=u.$\textbf{}\\
\textbf{Iterative step:} Given the current iterate $x^{k},$ calculate
the next iterate $x^{k+1}$ by 
\begin{equation}
x^{k+1}=\frac{1}{k+1}u+\frac{k}{k+1}P_{C_{m}}P_{C_{m-1}}\cdots P_{C_{1}}(x^{k}),\label{eq: Halpern=002013Wittman-algorithm}
\end{equation}
where $u$ is the given anchor point.
\par\end{flushleft}

\section{String-averaging methods for best approximation to the common fixed
points set of a family of firmly nonexpansive operators: The infinite
case \label{sec:The-infinite-case}}

In this section we propose a string-averaging method for solving the
best approximation problem to the common fixed point set of a countable
family of FNEs $(T_{i})_{i\in I},$ where $I$ is a countable set
of positive integers. We use similar terms to the ones that were used
in Section \ref{sec:The-finite-case} for the finite case, extending
Definition \ref{def:IndexVec-StringOp-FitOmega-FiniteCase} to the
countable case.

The following definition elaborates how this expansion is made.
\begin{defn}
\label{def:IndexVec-StingrOp-FitOmega-InfinteCase} Let $D$ be a
nonempty closed convex subset of $H$, let $(T_{i})_{i\in I}$ be
a countable family of self operators on $D.$ An\textit{ index vector}
is a vector of the form $t=(t_{1},t_{2},\ldots,t_{p})$ such that
$t_{\ell}\in I$ for all $\ell\in\left\{ 1,2,\ldots,p\right\} $.
For a given index vector $t=(t_{1},t_{2},\ldots,t_{q})$ we denote
its length (i.e., the number of its components) by $\gamma(t)=q$,
and define the operator $T[t]$ as the (finite) composition of the
operators $T_{i}$ whose indices appear in the index vector $t$,
namely,
\begin{equation}
T[t]:=T_{t_{q}}T_{t_{q-1}}\cdots T_{t_{1}},\label{eq:T=00005Bt=00005D}
\end{equation}
and call it a \textit{string operator}. An infinite set $\varOmega$
of index vectors is called \textit{fit} if for each $i\in I$, there
exist a vector $t=(t_{1},t_{2},\ldots,t_{p})\in\varOmega$ such that
$t_{\ell}=i$ for some $\ell\in\left\{ 1,2,\ldots,p\right\} $. Denote
by $\mathcal{M}$ the collection of all pairs $(\varOmega,w)$, where
$\varOmega$ is a\textit{ }fit \textbf{countable} set of index vectors
and $w:\varOmega\rightarrow(0,1)$ is such that $\sum_{t\in\Omega}w(t)=1$.
\end{defn}

Observe that in Definition \ref{def:IndexVec-StringOp-FitOmega-FiniteCase}
$w:\Omega\rightarrow(0,1]$ is permitted whereas here we must have
$w:\varOmega\rightarrow(0,1)$. This is due to the fact that in the
infinite case here it is impossible to put all operators $(T_{i})_{i\in I}$
in a single string operator $T[t]$.

\subsection{The static string-averaging method for the infinite case\label{subsec:staic-SA-infinite}}

Algorithm 1 handles a finite family of FNEs $(T_{i})_{i=1}^{m}$ and
throughout its iterative process, a finite number of string operators
$T\left[t\right]$ is used to construct a single string-averaging
operator $T=\sum_{t\in\Omega}w(t)T\left[t\right]$. In the infinite
case we allow an infinite family $(T_{i})_{i\in I}$. In our extension
to the infinite case $\varOmega$ is countable (not finite as it was
before) and we allow a countable number of string operators $T\left[t\right]$
and a chosen fixed (infinite dimensional) weight vector, making now
the single string-averaging operator $T:=\sum_{t\in\Omega}w(t)T\left[t\right]$
an infinite series. The fact that only a single $T$ is used makes
this a ``static'' string-averaging method. Algorithm 8 below is
our suggested method for this case.\medskip{}

\begin{flushleft}
\textbf{Algorithm 8.\label{alg:str-avg-InfiniteCase} The static string-averaging
algorithm for solving the best approximation problem in the countable
case.}\\
\textbf{Initialization: }Choose a single pair $(\varOmega,w)\in\mathcal{M}$
and an arbitrary $x^{0}\in D.$\textbf{}\\
\textbf{iterative step:} Given the current iterate $x^{k},$ calculate
the next iterate $x^{k+1}$ by
\begin{equation}
x^{k+1}=\lambda_{k}u+(1-\lambda_{k})\sum_{t\in\Omega}w(t)T\left[t\right]\left(x^{k}\right),\label{eq:Algorithm-9-formula}
\end{equation}
where $(\lambda_{k})_{k\in\mathbb{N}}$ is a steering sequence, $u$
is the given anchor point and $w(t)$ and $T[t]$ are as in Definition
\ref{def:IndexVec-StingrOp-FitOmega-InfinteCase}.\medskip{}
\par\end{flushleft}

Our aim is to use again Theorem \ref{thm-1-Halpern} above to prove
the convergence of Algorithm 8 to the projection of the anchor point
onto the common fixed point set of the initial given family of FNEs.
Recall that, when we did so earlier for the finite case, we used Lemma
\ref{lemma-2 - correct set}, which in turn, depended on Propositions
\ref{prop-1-Corollary 4.50 BC} and \ref{prop-2- Prop4.47}. As we
are now dealing with the countable case, we first extend these propositions
to the countable case. Note that since the length of any index vector
remains finite in our countable case, Proposition \ref{prop-1-Corollary 4.50 BC}
remains applicable. Hence, we only slightly adjust Proposition \ref{prop-2- Prop4.47}
as follows and present an elementary proof of it.
\begin{prop}
\label{prop-4- Prop4.47-infinite case}Let $D$ be a nonempty subset
of $H$, let $(T_{i})_{i\in I}$ be a countable family of self QNEs
on $D$ such that $\cap_{i\in I}Fix(T_{i})\neq\emptyset$ and let
$(w_{i})_{i\in I}$ be a countable sequence of strictly positive real
numbers such that $\sum_{i\in I}w_{i}=1$. Then $Fix(\sum_{i\in I}w_{i}T_{i})=\cap_{i\in I}Fix(T_{i})$.
\end{prop}

\begin{proof}
First, let us show that $\sum_{i\in I}w_{i}T_{i}$ is a well-defined
operator. That is, for every $x\in D$, we have to show that $\sum_{i\in I}w_{i}T_{i}(x)$
converges as an infinite series. Fix $x\in D$, let $f\in\cap_{i\in I}Fix(T_{i})$
and let $T_{i}\in(T_{i})_{i\in I}$. First, we observe that, by the
quasi-nonexpansivity of $T_{i}$, 
\begin{equation}
\Vert T_{i}(x)\Vert=\Vert T_{i}(x)-f+f\Vert\leq\Vert T_{i}(x)-f\Vert+\Vert f\Vert\leq\Vert x-f\Vert+\Vert f\Vert,\label{eq: 40}
\end{equation}
which shows that the series

\begin{equation}
\sum_{i\in I}w_{i}(\Vert x-f\Vert+\Vert f\Vert)=\Vert x-f\Vert+\Vert f\Vert\label{eq: 41}
\end{equation}
converges. Define the sequence of partial sums of $\sum_{i\in I}w_{i}T_{i}(x)$,
i.e., for every $n\in\mathbb{N},$ $S_{n}:=\sum_{i=1}^{n}w_{i}T_{i}(x)$.

In order to apply the Cauchy criterion for series convergence, we
use (\ref{eq: 40}) to obtain
\begin{equation}
\Vert S_{n+p}-S_{n}\Vert=\Vert\sum_{i=n+1}^{n+p}w_{i}T_{i}(x)\Vert\leq\sum_{i=n+1}^{n+p}w_{i}\Vert T_{i}(x)\Vert\leq\sum_{i=n+1}^{n+p}w_{i}(\Vert x-f\Vert+\Vert f\Vert).\label{eq: 42}
\end{equation}
Since the series in (\ref{eq: 41}) is convergent it follows that
\begin{equation}
\sum_{i=n+1}^{n+p}w_{i}(\Vert x-f\Vert+\Vert f\Vert)<\varepsilon,
\end{equation}
which shows that $(S_{n})_{n\in\mathbb{N}}$ is a Cauchy sequence,
thus, from the completeness of $H$, $(S_{n})_{n\in\mathbb{N}}$ converges
and, hence, $\sum_{i\in I}w_{i}T_{i}(x)$ converges as well and $\sum_{i\in I}w_{i}T_{i}(x$)
is well-defined, as required.

The rest of the proof is similar to the proof of Corollary 4.48 in
\cite{BauschkeCombettes}, which is quoted above in our Proposition
\ref{prop-2- Prop4.47}. Set $Q:=\sum_{i\in I}w_{i}T_{i}.$ It is
clear that $\cap_{i\in I}Fix(T_{i})\subseteq Fix(Q)$. To handle the
opposite inclusion, we first do the following calculation. Let $y\in\cap_{i\in I}Fix(T_{i})$.
Then, for every $i\in I$ and any $x\in D,$
\begin{equation}
\begin{aligned}\Vert T_{i}(x)-y\Vert^{2} & =\Vert(T_{i}(x)-x)+(x-y)\Vert^{2}\\
 & =\Vert T_{i}(x)-x\Vert^{2}+2\left\langle T_{i}(x)-x,x-y\right\rangle +\Vert x-y\Vert^{2},
\end{aligned}
\end{equation}
which, together with the quasi-nonexpansivity of $T_{i},$ yields
\begin{equation}
\begin{aligned} & 2\left\langle T_{i}(x)-x,x-y\right\rangle \\
 & =\Vert T_{i}(x)-y\Vert^{2}-\Vert T_{i}(x)-x\Vert^{2}-\Vert x-y\Vert^{2}\\
 & \leq\Vert x-y\Vert^{2}-\Vert T_{i}(x)-x\Vert^{2}-\Vert x-y\Vert^{2}\\
 & =-\Vert T_{i}(x)-x\Vert^{2}.
\end{aligned}
\label{eq:45}
\end{equation}
Now let us take a point $z\in Fix(Q)$ and observe that it can be
rewritten as $z=\sum_{i\in I}w_{i}z$. Then, from (\ref{eq:45}) it
follows that 
\begin{equation}
0=2\left\langle Q(z)-z,z-y\right\rangle =2\sum_{i\in I}w_{i}\left\langle T_{i}(z)-z,z-y\right\rangle \leq-\sum_{i\in I}w_{i}\Vert T_{i}(z)-z\Vert^{2},\label{eq:46}
\end{equation}
but only if the right-hand side series $-\sum_{i\in I}w_{i}\Vert T_{i}(z)-z\Vert^{2}$
is convergent. To prove this last claim we use (\ref{eq: 40}) and
the Cauchy-Schwarz inequality, to reach
\begin{equation}
\left\langle T_{i}(z),z\right\rangle \leq|\left\langle T_{i}(z),z\right\rangle |\leq\Vert T_{i}(z)\Vert\Vert z\Vert\leq(\Vert z-f\Vert+\Vert f\Vert)\Vert z\Vert,\label{eq:47}
\end{equation}
for any $f\in\cap_{i\in I}Fix(T_{i})$. Hence, by (\ref{eq:47}),
and, once more, by (\ref{eq: 40}), we have
\begin{equation}
\sum_{i\in I}w_{i}\Vert T_{i}(z)-z\Vert^{2}=\sum_{i\in I}w_{i}\Vert T_{i}(z)\Vert^{2}-2\sum_{i\in I}w_{i}\left\langle T_{i}(z),z\right\rangle +\sum_{i\in I}w_{i}\Vert z\Vert^{2}.\label{eq:48}
\end{equation}
The first infinite sum on the right-hand side of (\ref{eq:48}) is
a convergent series due to (\ref{eq: 40}), the middle series on the
right-hand side of (\ref{eq:47}) is convergent because of (\ref{eq:47}),
and the last series converges trivially, thus, the series $\sum_{i\in I}w_{i}\Vert T_{i}(z)-z\Vert^{2}$
is convergent and (\ref{eq:46}) guarantees that $\sum_{i\in I}w_{i}\Vert T_{i}(z)-z\Vert^{2}=0.$
Since $w_{i}\neq0$ for all $i\in I,$ we obtain that $z\in\cap_{i\in I}Fix(T_{i})$.
\end{proof}
We are now ready to extend Lemma \ref{lemma-2 - correct set} to the
countable case.
\begin{lem}
\textbf{\label{lemma-3-correct set infinite}}Let $D$ be a nonempty
closed convex subset of $H$ and let $(T_{i})_{i\in I}$ be a countable
family of self FNEs on $D$ such that $F:=\cap_{i\in I}Fix(T_{i})\neq\emptyset$.
Let the single pair $(\varOmega,w)\in\mathcal{\mathcal{M}}$ and let
$T:=\sum_{t\in\Omega}w(t)T[t]$ be the string-averaging operator.
Then $Fix(T)=\cap_{i\in I}Fix(T_{i})$.
\end{lem}

\begin{proof}
The proof proceeds in several steps. First, the family $(T[t])_{t\in\varOmega}$
is countable because $\varOmega$ is, and since $T[t]=T_{t_{q}}T_{t_{q-1}}\cdots T_{t_{1}}$
is a finite composition we use similar arguments to those in the proof
of Lemma \ref{lemma-2 - correct set}, to show that $(T[t])_{t\in\varOmega}$
is a family of QNEs with nonempty common fixed points set. Note that
$F\neq\emptyset$ and (\ref{eq:T=00005Bt=00005D}) imply that for
every $t\in\varOmega$ and for every $x\in F$, $T[t](x)=x$. Hence,
\begin{equation}
F\subseteq Fix(T[t]).\label{eq:Lemma2-1-1}
\end{equation}
Therefore,
\begin{equation}
F\subseteq\cap_{t\in\Omega}Fix(T[t]).\label{eq:Lemma2-2-1}
\end{equation}
Since every FNE is NE (see Lemma \ref{lemma-1- link}(i)), we apply,
for every $t\in\varOmega,$ \cite[Lemma 2.1.12(ii)]{Cegielski} to
the family $(T_{t_{\ell}})_{\ell=1}^{q}$ and conclude that the string
operator $T[t]:=T_{t_{q}}T_{t_{q-1}}\cdots T_{t_{1}},$ is NE. Thus,
from (\ref{eq:Lemma2-1-1}) which guarantees the non-emptiness of
$Fix(T[t])$ and due to the fact that every NE with a fixed point
is QNE, the family of operators $(T[t])_{t\in\varOmega}$ is a family
of QNEs.

Secondly, by similar arguments to those that appear in the first part
of the proof of Proposition \ref{prop-4- Prop4.47-infinite case},
$T$ is a well-defined operator and, we can apply Proposition \ref{prop-4- Prop4.47-infinite case}
with the family $(T[t])_{t\in\varOmega}$ instead of the family $(T_{i})_{i\in I}$
and with the family $(w(t))_{t\in\varOmega}$ instead of the family
$(w_{i})_{i\in I}$, used there. Hence, we get

\begin{equation}
Fix(T)=\cap_{t\in\varOmega}Fix(T[t]).\label{eq:51}
\end{equation}

Now we apply Proposition \ref{prop-1-Corollary 4.50 BC} to the finite
family $(T_{t_{\ell}})_{\ell=1}^{q}.$ To do so we note that, by Lemma
\ref{lemma-1- link}(iii), all members of $(T_{i})_{i\in I}$ are
sQNEs and, so, for every $t\in\varOmega$, the Proposition \ref{prop-1-Corollary 4.50 BC}
yields that the string operator $T[t]$ satisfies 
\begin{equation}
Fix(T[t])=\cap_{\ell=1}^{q}Fix(T_{t_{\ell}}).\label{eq:Lemma2-3-1}
\end{equation}
 Finally, from the fitness of $\varOmega$, from (\ref{eq:51}) and
from (\ref{eq:Lemma2-3-1}), we obtain that 
\begin{equation}
Fix(T)=\cap_{t\in\varOmega}Fix(T[t])=\cap_{t\in\varOmega}(\cap_{\ell=1}^{q}Fix(T_{t_{\ell}}))=\cap_{i\in I}Fix(T_{i}).
\end{equation}
\end{proof}
It is well-known that a convex combination of NEs is a NE (see, e.g,
\cite[Lemma 2.1.12]{Cegielski}). Since the string operators $T[t]$
in Algorithm 8 appear inside an infinite series that resembles a convex
combination, our next aim is to show that, under certain assumptions,
the string-averaging operator $T=\sum_{t\in\Omega}w(t)T[t]$ is NE.
The following lemma will show this.
\begin{lem}
\label{lemma-5- infinite convex combination}Let $D$ be a nonempty
subset of $H$ and let $(S_{i})_{i\in I}$ be a countable family of
NEs such that, for every $i\in I$, $S_{i}:D\rightarrow H$ and $\cap_{i\in I}Fix(S_{i})\neq\emptyset.$
Then $S:=\sum_{i\in I}w_{i}S_{i}$, where $(w_{i})_{i\in I}$ is a
countable sequence of strictly positive real numbers such that $\sum_{i\in I}w_{i}=1$,
is NE.
\end{lem}

\begin{proof}
By using similar arguments to the ones that were used in Proposition
\ref{prop-4- Prop4.47-infinite case} for $T$, we deduce that $S$
is well-defined. Now, let $x,y\in D$. By the nonexpansivity of every
$S_{i}$ and since $\sum_{i\in I}w_{i}=1$, it follows that
\begin{equation}
\begin{aligned}||S(x)-S(y)|| & =||\sum_{i\in I}w_{i}S_{i}(x)-\sum_{i\in I}w_{i}S_{i}(y)||\leq\sum_{i\in I}||w_{i}S_{i}(x)-w_{i}S_{i}(y)||\\
 & \leq\sum_{i\in I}w_{i}||S_{i}(x)-S_{i}(y)||\leq\sum_{i\in I}w_{i}||x-y||=||x-y||.
\end{aligned}
\end{equation}
Thus, by Definition \ref{def: nonexpansive-class-operators}(i), $S$
is NE.
\end{proof}
We are now ready to prove convergence of Algorithm 8.
\begin{thm}
\label{thm-7-InfiniteCase-SA}Let $D$ be a nonempty closed convex
subset of $H$ and let $(T_{i})_{i\in I}$ be a countable family of
self FNEs on $D$ such that $F:=\cap_{i\in I}Fix(T_{i})\neq\emptyset$.
Let $(\Omega,w)\in\mathcal{\mathcal{M}}$, let $(\lambda_{k})_{k\in\mathbb{N}}$
be a steering sequence and let $u,x^{0}\in D$. Then any sequence
$(x^{k})_{k\in\mathbb{N}}$, generated by Algorithm 8, converges strongly
to $P_{F}(u)$.
\end{thm}

\begin{proof}
For $(T_{i})_{i\in I}$ and $(\Omega,w)\in\mathcal{\mathcal{M}}$
consider the family of operators $(T[t])_{t\in\varOmega}$ and define
the string-averaging operator $T:=\sum_{t\in\varOmega}w(t)T\left[t\right]$.
By the proof of Lemma \ref{lemma-2 - correct set}, we deduce that
$(T[t])_{t\in\varOmega}$ is a countable family of NEs and, thus,
$T$ is a ``convex combination'' of countably many NEs. Since $F\subseteq\cap_{t\in\varOmega}Fix(T[t])$,
it follows that $\cap_{t\in\varOmega}Fix(T[t])\neq\emptyset$ and,
therefore, it can be shown by using similar arguments to the ones
that were used in Proposition \ref{prop-4- Prop4.47-infinite case},
that $T$ is a well-defined operator. Moreover, Lemma \ref{lemma-5- infinite convex combination}
yields that $T$ is NE, and since $Fix(T)\supseteq F\neq\emptyset$,
Theorem \ref{thm-1-Halpern} is applicable to $T$ and implies that
every sequence $(x^{k})_{k\in\mathbb{N}}$ generated by Algorithm
8 converges strongly to $P_{Fix(T)}(u)$. Now, by applying Lemma \ref{lemma-3-correct set infinite}
on $(T_{i})_{i\in I}$ together with $(\varOmega,w)$, we obtain that
$Fix(T)=F$, thus, $x^{k}\rightarrow P_{F}(u).$
\end{proof}
In \cite{Combettes} the following simultaneous algorithm for solving
the best approximation problem to the common fixed point set of a
countable family of FNEs with a nonempty common fixed points set is
studied. This algorithm turns out to be a special case of our Algorithm
8 when the pair $(\varOmega,w)$ is chosen $\Omega:=\left\{ (1),(2),(3),\ldots\right\} $
and for every $i\in I$, $w((i)):=w_{i}$. The convergence of this
algorithm then follows from our Theorem \ref{thm-7-InfiniteCase-SA}.\medskip{}

\begin{flushleft}
\textbf{Algorithm 9. \label{alg:Combettes} Combettes simultaneous
algorithm.}\\
\textbf{Initialization: }Choose $x^{0}\in D.$\textbf{}\\
\textbf{iterative step:} Given the current iterate $x^{k},$ calculate
the next iterate $x^{k+1}$ by
\begin{equation}
x^{k+1}=\lambda_{k}u+(1-\lambda_{k})\sum_{i\in I}w_{i}T_{i}(x^{k}),\label{eq: Algorithm-Combettes}
\end{equation}
where $(\lambda_{k})_{k\in\mathbb{N}}$ is a steering sequence, $u$
is the given anchor point and $(w_{i})_{i\in I}$ is a sequence of
user-chosen strictly positive real numbers such that $\sum_{i\in I}w_{i}=1$.
\par\end{flushleft}

\section{Concluding comments\label{sec:Concluding-comments}}

In this work we expand the class of string-averaging methods by proposing
new string-averaging approaches for solving the BAP for the common
fixed points set of either a finite or an infinite family of FNEs.
These methods vary from the ``static'' format, in which one uses
a single pair $(\varOmega,w)\in\mathcal{\mathcal{M}},$ to quasi dynamic
and simultaneous formats, where more than a single pair is used. Nevertheless,
these methods are obtained with string operators which are defined
as finite compositions of operators, taken from the given family of
operators, thus, the question for infinitely many compositions, i.e.,
strings of infinite length, remains open, as well as the question
whether our methods can be extended to the fully dynamic case.

Another question that remains open is whether the results of \cite{Aleyner},
where a sequential algorithm for solving the BAP to the common fixed
points set of a semigroup of nonexpansive operators in Hilbert space
was studied, can be extended to encompass string-averaging algorithmic
schemes.

\section{Declarations}

Availability of data and materials: Not applicable.

Competing interests: The authors declare that they have no competing
interests.

Funding: This work was supported by the ISF-NSFC joint research program
grant No. 2874/19.

Authors' contributions: All authors contributed equally and significantly
in this research work. All authors read and approved the final manuscript.

Acknowledgements: We are grateful to Shoham Sabach and Alexander Zaslavski
for their very useful comments on an earlier version of this work.
We thank the Referee and the Editor for their comments.

\addcontentsline{toc}{section}{References}

\begin{newpage}

{\large{}\bibliographystyle{plain}
\bibliography{bib-031120}

\begin{thebibliography}{10}

\bibitem{Aleyner}
A.~Aleyner and Y.~Censor.
\newblock Best approximation to common fixed points of a semigroup of
  nonexpansive operators.
\newblock {\em Journal of Nonlinear and Convex Analysis}, 6:137--151, 2005.

\bibitem{Artacho}
F.J.~Arag{\'o}n Artacho and R.~Campoy.
\newblock A new projection method for finding the closest point in the
  intersection of convex sets.
\newblock {\em Computational Optimization and Applications}, 69:99--132, 2018.

\bibitem{Bargetz}
C.~Bargetz, S.~Reich, and R.~Zalas.
\newblock Convergence properties of dynamic string-averaging projection methods
  in the presence of perturbations.
\newblock {\em Numerical Algorithms}, 77:185--209, 2018.

\bibitem{HLWB}
H.H. Bauschke.
\newblock The approximation of fixed points of comositions of nonexpansive
  mappings in {H}ilbert space.
\newblock {\em Journal of Mathematical Analysis and Applications},
  202:150--159, 1996.

\bibitem{BauschkeBorweinAlternatingProjection}
H.H. Bauschke and J.M Borwein.
\newblock On the convergence of von {N}eumann's alternating projection
  algorithm for two sets.
\newblock {\em Set-Valued Analysis}, 1:185--212, 1993.

\bibitem{BauschkeBorwein}
H.H. Bauschke and J.M. Borwein.
\newblock On projection algorithms for solving convex feasibility problems.
\newblock {\em SIAM Review}, 38:367--426, 1996.

\bibitem{BauschkeCombettes}
H.H. Bauschke and P.L. Combettes.
\newblock {\em Convex Analysis and Monotone Operator Theory in Hilbert Space}.
\newblock Springer International Publishing AG, 2nd edition, 2017.

\bibitem{BauschkeKoch}
H.H. Bauschke and V.R. Koch.
\newblock Projection methods: Swiss army knives for solving feasibility and
  best approximation problems with halfspaces.
\newblock {\em Contemporary Mathematics}, 636:1--40, 2015.

\bibitem{ReichProjectionsHyperplanes}
L.M. Bergman, Y.~Censor, S.~Reich, and Y.~Zepkowitz-Malachi.
\newblock Finding the projection of a point onto the intersection of convex
  sets via projections onto half-spaces.
\newblock {\em Journal of Approximation Theory}, 124:194--218, 2003.

\bibitem{blat2006}
D.~Blat and A.O.~Hero III.
\newblock Energy based sensor network source localization via projection onto
  convex sets ({P}{O}{C}{S}).
\newblock {\em IEEE Transactions on Signal Processing}, 54:3614--3619, 2006.

\bibitem{BoyleDykstra}
J.P. Boyle and R.L. Dykstra.
\newblock A method for finding projections onto the intersection of convex sets
  in {H}ilbert spaces.
\newblock In R.~Dykstra, T.~Robertson, and F.T. Wright, editors, {\em Advances
  in Order Restricted Statistical Inference}, pages 28--47. Springer, Lecture
  Notes in Statistics, vol. 37, New York, NY, USA, 1986.

\bibitem{CQmethods}
M.~Brooke, Y.~Censor, and A.~Gibali.
\newblock Dynamic string-averaging {C}{Q}-methods for the split feasibility
  problem with percentage violation constraints arising in radiation therapy
  treatment planning.
\newblock Technical report, https://arxiv.org/abs/1911.12041., 2019.

\bibitem{Cegielski}
A.~Cegielski.
\newblock {\em Iterative Methods for Fixed Point Problems in Hilbert Spaces}.
\newblock Springer-Verlag Berlin Heidelberg, 2012.

\bibitem{Censor}
Y.~Censor.
\newblock Computational acceleration of projection algorithms for the linear
  best approximation problem.
\newblock {\em Linear Algebra and its Applications}, 416:111--123, 2006.

\bibitem{CensorCegielski}
Y.~Censor and A.~Cegielski.
\newblock Projection methods: an annotated bibliography of books and reviews.
\newblock {\em Optimization}, 64:2343--2358, 2015.

\bibitem{Pink-book}
Y.~Censor, T.~Elfving, and G.T. Herman.
\newblock Averaging strings of sequential iterations for convex feasibility
  problems.
\newblock In D.~Butnariu, Y.~Censor, and S.~Reich, editors, {\em Inherently
  Parallel Algorithms in Feasibility and Optimization and Their Applications},
  pages 101--114. Elsevier Science Publishers, Amsterdam, The Netherlands,
  2001.

\bibitem{SegalCensor}
Y.~Censor and A.~Segal.
\newblock On the string averaging method for sparse common fixed-point
  problems.
\newblock {\em International Transactions in Operational Research},
  16:481--494, 2009.

\bibitem{CensorTom}
Y.~Censor and E.~Tom.
\newblock Convergence of string-averaging projection schemes for inconsistent
  convex feasibility problems.
\newblock {\em Optimization Methods and Software}, 18:543--554, 2003.

\bibitem{CensorZaslavski1}
Y.~Censor and A.J. Zaslavski.
\newblock Convergence and perturbation resillence of dynamic string-averaging
  projection methods.
\newblock {\em Computational Optimization and Applications}, 54:65--76, 2013.

\bibitem{CensorZaslavski2}
Y.~Censor and A.J. Zaslavski.
\newblock String-averaging projected subgradient methods for constrained
  minimization.
\newblock {\em Optimization Methods and Software}, 29:658--670, 2014.

\bibitem{Combettes}
P.L. Combettes.
\newblock Construction d'un point fixe commun {\`a} une famille de contractions
  fermes.
\newblock {\em Comptes Rendus de l'Acad{\'e}mie des Sciences de Paris,
  S{\'e}rie A (Math{\'e}matique)}, 320:1385--1390, 1995.

\bibitem{Crombez}
G.~Crombez.
\newblock Finding common fixed points of strict paracontractions by averaging
  strings of sequential iterations.
\newblock {\em Journal of Nonlinear and Convex Analysis}, 3:345--351, 2002.

\bibitem{DeutchAlternatingProjections}
F.~Deutsch.
\newblock Rate of convergence of the method of alternating projections.
\newblock In B.~Brosowski and F.~Deutsch, editors, {\em Parametric Optimization
  and Approximation}, pages 96--107. International Series of Numerical
  Mathematics, vol. 72. Birkh{\"a}user, Basel, Switzerland, 1984.

\bibitem{DeutschBook}
F.~Deutsch.
\newblock {\em Best Approximation in Inner Product Spaces}.
\newblock Springer-verlag, New-York, 2001.

\bibitem{DeutschHundal}
F.~Deutsch and H.~Hundal.
\newblock The rate of convergence for the method of alternating projections
  {II}.
\newblock {\em Journal of Mathematical Analysis and Applications},
  205:381--405, 1997.

\bibitem{YamadaDeutsch}
F.~Deutsch and I.~Yamada.
\newblock Minimizing certain convex functions over the intersection of the
  fixed point sets of nonexpansive mappings.
\newblock {\em Numerical Functional Analysis and Optimization}, 19:33--56,
  1998.

\bibitem{reich1991}
J.~Dye, M.A. Khamsi, and S.~Reich.
\newblock Random products of contractions in {B}anach spaces.
\newblock {\em Transactions of the American Mathematical Society}, 325:87--99,
  1991.

\bibitem{Dykstra}
R.L. Dykstra.
\newblock An algorithm for restricted least squares regression.
\newblock {\em Journal of the American Statistical Association}, 78:837--842,
  1983.

\bibitem{Halpern}
B.~Halpern.
\newblock Fixed points of nonexpanding maps.
\newblock {\em Bulletin of the American Mathematical Society}, 73:957--961,
  1967.

\bibitem{kong2019}
T.Y. Kong, H.~Pajoohesh, and G.T. Herman.
\newblock String-averaging algorithms for convex feasibility with infinitely
  many sets.
\newblock {\em Inverse Problems}, 35, 2019.
\newblock 045011.

\bibitem{ReichKopecka}
E.~Kopeck{\'a} and S.~Reich.
\newblock A note on the von {N}eumann alternating projections algorithm.
\newblock {\em Journal of Nonlinear and Convex Analysis}, 5:379--386, 2004.

\bibitem{Lopez-Xu}
G.~L{\'o}pez, V.~Martin-M{\'a}rquez, and H.~Xu.
\newblock Halpern's iteration for nonexpansive mappings.
\newblock {\em Contemporary Mathematics}, 513:211--231, 2010.

\bibitem{ReichZalas}
S.~Reich and R.~Zalas.
\newblock A modular string averaging procedure for solving the common fixed
  point problem for quasi-nonexpansive mappings in {H}ilbert space.
\newblock {\em Numerical Algorithms}, 72:297--323, 2016.

\bibitem{NeumanBook}
J.~{v}on Neumann.
\newblock {\em Functional Operators II: The Geometry of Orthogonal Spaces}.
\newblock Princeton University Press, 1950.
\newblock Reprint of mimeographed lecture notes first distributed in 1933.

\bibitem{Wittman}
R.~Wittman.
\newblock Approximation of fixed points of nonexpansive mappings.
\newblock {\em Archiv der Mathematik}, 58:486--491, 1992.

\end{thebibliography}
}{\large\par}

\end{newpage}
\end{document}